\documentclass{birkjour}
 \newtheorem{thm}{Theorem}[section]
 \newtheorem{cor}[thm]{Corollary}
 \newtheorem{lem}[thm]{Lemma}
 
 \theoremstyle{definition}
 \newtheorem{defn}[thm]{Definition}
 \theoremstyle{remark}

 \numberwithin{equation}{section}

\begin{document}

%
%
%
%
%
%
%
%
%

\title[The Fermi-Walker Derivative in Galilean Space]
 {The Fermi-Walker Derivative in Galilean Space}

\author{Tevfik \c{S}AH\.{I}N}
\address{%
	Department of Mathematics,\\
	Faculty of Arts and Sciences,\\
	Amasya University,\\
	05000 Amasya\\
	Turkey}
\email{tevfik.sahin@amasya.edu.tr, tevfiksah@gmail.com}

\author[Fatma KARAKU\c{S}]{Fatma KARAKU\c{S}}

\address{%
	Department of Mathematics,\\
	Faculty of Arts and Sciences\\
	Sinop University,\\
	57000 Sinop\\
	Turkey
}

\email{fkarakus@sinop.edu.tr}
\author{Keziban ORBAY}
\address{Mathematics and Science Education Department,\br
Education Faculty,\br
Amasya University,\br
05000 Amasya\br
Turkey}
\email{keziban.orbay@amasya.edu.tr}


\date{January 1, 2004}

\begin{abstract}
In this study, we defined   Fermi-Walker derivative in Galilean space $\mathbb{G}^3$. Fermi-Walker transport and
non-rotating frame by using Fermi-Walker derivative are given in $\mathbb{G}^3$. Being conditions of Fermi-Walker transport and non-rotating frame are investigated along any curve for Frenet frame and Darboux frame. 
\end{abstract}

\maketitle
\section{Introduction}
The universe is perceptible through observation. A relativistic observer $%
\gamma $ needs reference frames, for measurements of durations and of
(\textquotedblright geo\textquotedblright )metric quantities: the proper
time (\textquotedblright proper clock\textquotedblright ) is given by its
own canonical parameter running on an interval of the real numbers axis; the
restspaces are refered to \textquotedblright fixed\textquotedblright\
directions, maintained by gyroscopes focused toward \textquotedblright
fixed\textquotedblright\ celestial bodies.

The choice of an appropriate reference frame is a fundamental and
controversial problem in astronomy \cite{8}: one needs a \textquotedblright
center\textquotedblright\ and several \textquotedblright
fixed\textquotedblright\ directions. In a general relativistic setting, if $%
\gamma $ is freely falling, its restspaces are transported through
Levi-Civita parallelism, so a fix spacelike direction has, by definition, a
null covariant derivative \cite{9,10}. If $\gamma $ is not freely falling
i.e. for accelerated observes, the restspace are not transported by the
Levi-Civita parallelism, anymore. In this case, in order to define
"constant" directions, another parallelism is used: the Fermi-Walker
transport which is an isometry between the tangent space along $\gamma $ 
\cite{4,5,11,13,14}. Fermi-Walker transport is a process used to define a
coordinate system or reference frame in general relativity. All the
curvatures in the reference frame in due to the presence of mass-energy
density. These curvatures are not arbitrary spin or rotation of the frame.

"There are different transport laws such as parallel and Fermi-Walker
transport for a tensor along a given curve. The parallel transport for the
tensor along the given curve is defined as the law which makes that its
covariant derivative be zero \cite{3}. If the curve is a geodesic, then the
tangent vector will coincide at another point of the curve with its parallel
transported vector. Otherwise, the tangent vector will not coincide with its
parallel transported vector. In this case, there is Fermi-Walker's law that
another transport law. The Fermi-Walker transport of the tensor along the
given curve is defined as the law which makes that its the Fermi-derivative
along the curve be zero \cite{3}. If the curve is a geodesic, then
Fermi-Walker's transport coincides with parallel transport. Otherwise, this
is not the case. In general, the Fermi-Walker transport is not the parallel
transport"\cite{17}.

"A Fermi-Walker transported set of tetrad fields is the best approximation to
a non-rotating reference frame in the sense of Newtonian mechanics. It is
physically realized by a system of gyroscopes. Fermi-Walker transported
frames are important in lots of investigations. A frame that undergoes
linear and rotational acceleration can be described by the Frenet-Serret
frame. The relative rotational acceleration of a Frenet-Serret frame with
respect to a Fermi-Walker transported frame is taken to characterize
important phenomena, like the gyroscopic precession \cite{19}. Non-inertial
reference frames in Minkowski spacetime that undergo Fermi-Walker transport
are useful, for example, in the analysis of the inertial effects on a Dirac
particle" \cite{20}.

Parallel vector fields have important applications in differential geometry,
physics and especially in robotic kinematics. The tangent vector of the
curve is parallel along the curve if and only if $\nabla _{T}T=0$ in
Euclidean space. In this case, the curve is a geodesic in Euclidean space $%
\mathbb{E}^{n}$. Similarly, the curve which is on the surface is a geodesic
if and only if $\overline{\nabla }_{T}T=0$. Namely, the tangent vector is
parallel along the curve on the surface. All the straight lines are geodesic
curves in Euclidean space. I wonder if all the curves will be geodesic in
Euclidean space? The answer to this is hidden in the connection which is
obtained by using Fermi-Walker derivative. Indeed, the solution of $%
\widetilde{\nabla }_{T}T=0$, is provided for all curves in $\mathbb{E}^{n}$.
Accordingly, the curves and the lines are the same. That is, the curves
behave like the lines with respect to Fermi-Walker connection which is a
affine connection. \ \ 

In \cite{1,2}, Fermi--Walker derivative along any space curve was identified
and was given physical properties in $\mathbb{E}^{3}$.

In \cite{7}, Fermi-Walker derivative, Fermi-Walker transport and
non-rotating frame are analyzed for Bishop, Darboux and Frenet frames along
the curve in Euclidean space.

In \cite{16}, we have shown Fermi-Walker derivative, and non-rotating frame
are being conditions are analyzed in Minkowski space $\mathbb{E}_{1}^{3}.$

In \cite{17}, Fermi-Walker derivative is redefined in dual space $\mathbb{D}%
^{3}$. Fermi-Walker transport and non-rotating frame being conditions are
analyzed along the dual curves in dual space $\mathbb{D}^{3}$.

The notion of Fermi-Walker derivative, it shows us one method, which is used
for defining \textquotedblright constant\textquotedblright direction, that
may contain lots of condition to have Fermi-Walker transport or non-rotating
frame. The condition of Fermi-Walker transport depends on a solution that
contains differential equation system which is not always easy to find the
answer. Therefore, it is important to analyze this concept. In this paper,
Fermi-Walker derivative, Fermi-Walker transport and non-rotating frame
concepts are defined along any curve and the notions have been analyzed for
both isotropic and non-isotropic vector fields.

We have investigated Fermi-Walker derivative and geometric applications in
various spaces like Euclidean, Lorentz and Dual space up to now \cite%
{7,16,17}. The Fermi-Walker derivative which is defined in $\mathbb{G}^{3}$
is different from them so far since it is examined for both isotropic and
non-isotropic vector fields.

Firstly, Fermi-Walker derivative is redefined for any isotropic vector
fields along a curve which is in Galilean space. We have proved Fermi-Walker
derivative that is defined for the isotropic vector fields coincides with
Fermi derivative which is defined in any surface. We have shown Fermi-Walker
derivative which is defined for any non-isotropic vector fields is not
coincides with derivative of the vector fields. Being Fermi-Walker transport
conditions are examined for any isotropic and non-isotropic vector fields.
We have shown that if the curve is a line or a planar curve which is not a
line then the non-zero isotropic vector field is Fermi-Walker transported.
We have obtained that Frenet frame is not a non-rotating frame if the curve
is not a line. 

Then, similar investigations have been made for any isotropic and
non-isotropic vector fields with respect to the Darboux frame in Galilean
space. We have proved while the curve is a line the Darboux frame is a
non-rotating frame.

\section{Preliminaries}

In non-homogeneous coordinates the group of motion of $3$ - dimensional Galilean Geometry (i.e. the group of isometries of $\mathbb{G}^3$) has the form
define:
\begin{eqnarray}
\overline{x} &=&a_1+x,  \notag \\
\overline{y} &=&a_2+a_3x+y\cos \varphi +z\sin \varphi , \\
\overline{z} &=&a_4+a_5x-y\sin \varphi +z\cos \varphi,  \notag
\end{eqnarray}%
where $a_1, a_2, a_3, a_4, a_5$, and $\varphi$ are real numbers \cite{pav}. 

If the first component of a vector is not zero, then the vector is called as non-isotropic, otherwise it is called isotropic vector \cite{pav}.

The scalar product of two vectors $\mathbf{v}=(v_{1},v_{2},v_{3})$ and $\mathbf{w}=(w_{1},w_{2},w_{3})$ in $\mathbb{G}^3$ is defined by
\begin{equation}{\label{galpro}}
\langle\mathbf{v}, \mathbf{w} \rangle= \left\{
\begin{array}{lr}
v_{1}w_{1} , &  \text{if } v_{1}\neq 0 \text{ or } w_{1}\neq 0\, \ \  \ \\
v_{2}w_{2}+v_{3}w_{3} ,&  \text{if } v_{1}=0 \text{ and } w_{1}=0\,.
\end{array}\right.
\end{equation}
If $\langle\mathbf{v}, \mathbf{w} \rangle=0$, then $\mathbf{v}$ and $\mathbf{w}$
are perpendicular. 
The norm of $\mathbf{w}$ is defined by
$$\Vert \mathbf{w}\Vert_{G}=\sqrt{\langle\mathbf{w}, \mathbf{w} \rangle}.$$
Also, the Galilean cross product of two vectors defined by
\begin{equation}
\mathbf v\times _{G}\mathbf w=%
\begin{vmatrix}
0 & \mathbf e_{2} &\mathbf {e_{3}} \\ 
v_{1} & v_{2} & v_{3} \\ 
w_{1} & w_{2} & w_{3}%
\end{vmatrix}%
\end{equation}%
\newline
for $\mathbf{v=}\left( v_{1},v_{2},v_{3}\right) $ and $\mathbf{w=}\left(
w_{1},w_{2},w_{3}\right) $ \cite{pav1}. 

Let  $\alpha :I\subset \mathbb R\rightarrow \mathbb{G}^3$ be a curve parameterized by arc length (we abbreviate as p.b.a.l) with curvature $\kappa>0$ and torsion $\tau$.
If $\alpha$ is a unit speed curve,
\begin{equation*}
\alpha \left( x\right) =\left( x,y\left( x\right) ,z\left( x\right) \right) ,
\end{equation*}%
then the Frenet frame fields are given by
\begin{eqnarray}
T\left(x\right) &=&\alpha ^{\prime }\left( x\right), 
\notag \\
N\left( x\right) &=& \frac{\alpha''(x)}{\Vert \alpha''(x)\Vert_{G}}
\\
B\left( x\right) &=&T(x)\times _{G}N(x) \\&=&\frac{1}{\kappa \left( x\right) }\left( 0,
-z^{\prime \prime }\left( x\right) , y^{\prime \prime }\left(
x\right) \right) ,  \notag
\end{eqnarray}%
where $\kappa \left( x\right) $ and $\tau \left(
x\right) $ are defined by%
\begin{equation}
\kappa \left( x\right) ={\Vert \alpha''(x) \Vert }_{G}, { \ \ }\tau
\left( x\right) =\frac{\det \left( \alpha ^{\prime }\left( x\right) ,\alpha
	^{\prime \prime }\left( x\right) ,\alpha ^{\prime \prime \prime }\left(
	x\right) \right) }{\kappa ^{2}\left( x\right) }\,.
\end{equation}%

The vectors $T, N $ and $B$ are called the vectors of the tangent, the principal normal
and the binormal vector field, respectively \cite{pav1}. Therefore, the Frenet-Serret formulae can be written as
\begin{equation}
\begin{bmatrix}
T \\ 
N \\ 
B%
\end{bmatrix}%
^{\prime }=%
\begin{bmatrix}
0 & \kappa & 0 \\ 
0 & 0 & \tau \\ 
0 &- \tau & 0%
\end{bmatrix}%
\begin{bmatrix}
T \\ 
N \\ 
B%
\end{bmatrix}\,.
\end{equation}%

\begin{thm}
For any curve $\alpha :I\subset \mathbb R\rightarrow \mathbb{G}^{3}$, we call $D\left( x\right) =\tau
\left( x\right) T\left( x\right) +\kappa \left( x\right) B\left( x\right) $
a \textit{Darboux vector of }$\alpha$ \cite{tsam}.   By using the darboux vector,
Frenet-Serret formulas can be rewritten as follows:
\begin{eqnarray}
T'\left( x\right) &=&D\left( x\right) \times _{G} T\left( x\right)  \notag \\
N'\left( x\right) &=&D\left( x\right) \times _{G} N\left( x\right) \\
B'\left( x\right) &=&D\left( x\right) \times _{G} B\left( x\right)  \notag.
\end{eqnarray}
\end{thm} 
We define a vector $\widetilde{D}\left( x\right) =\left( \frac{\tau }{%
	\kappa }\right) \left( x\right) t\left( x\right) +b\left( x\right) $ and we
call it a modified Darboux vector along $\alpha.$

For more on Galilean Geometry, one can refer to \cite{pav, pav1, ros, yag} and references there in.

\begin{defn}\label{fermi dfn}
	$X$ is any vector field and $\alpha$ is
	unit-speed any curve in Galile space, then%
	\begin{equation}\label{fermi drv}
	\widetilde{\nabla }_{{T}}{X}=\nabla _{{T}}%
{X}-\left\langle {T},{X}\right\rangle 
{A}+\left\langle {A},{X}\right\rangle{T}
	\end{equation}%
	defined as $\widetilde{\nabla }_{{T}}{X}$ derivative is
	called Fermi-Walker derivative in Galilean space $\mathbb{G}^3$. Here ${T}$ is the tangent vector
	field of ${\alpha} $ and ${A}=\nabla _{{T}}{T}$
\end{defn}

\begin{defn}
	In Galilean space $\mathbb{G}^{3}$, let ${\alpha}:I\subset \mathbb{R}\rightarrow \mathbb{G}^{3}$ be a
	curve and ${X}$ be any vector field along the curve ${
		\alpha}$. If the Fermi-Walker derivative of the vector field ${X}$ vanishes, i.e., if $\widetilde{\nabla}_{{T}}{X}=0$, then ${X}$ is called the Fermi--Walker transported vector field along the curve.
\end{defn}

\begin{defn}
	Let a unit speed curve ${\alpha}:I\subset \mathbb{R}\rightarrow\mathbb{G}^{3}$ together
	with orthonormal vector field ${U},{V},{W}$
	along ${\alpha}$ be given. If the Fermi-Walker derivative of the
	vector field vanish, then $\left\{ U,V,W\right\} $ is called non-rotating
	frame.
\end{defn}

%

\section{Frenet Frame and Fermi--Walker Derivative}

In this section, Fermi-Walker derivative, Fermi-Walker transport and non-rotating frame concepts have been investigated along any curve which is in Galilean space. Fermi-Walker derivative has been redefined along a curve for both isotropic and non-isotropic vector fields. The vector fields which are Fermi-Walker transported are analyzed along any curve in $\mathbb{G}^3$. Then; we show that the Frenet frame whether it is a non-rotating frame or not.

\begin{lem}
	\label{L1}Let ${\alpha}:I\subset \mathbb{R}\rightarrow\mathbb{G}^{3}$ be a curve in
	Galilean space $\mathbb{G}^{3}$ and ${X}$ is any vector field along the curve $
	{\alpha}({x})$, Fermi-Walker derivative can be expressed as
	\begin{enumerate}
		\item[i)] If $X$ is an isotropic vector field along the $
		{\alpha}({x})$, then Fermi-Walker derivative of $X$ is given by 
	\begin{equation*}
	\widetilde{\nabla }_{{T}}{X}=\nabla _{{T}}%
	{X}+\kappa{\langle N, X\rangle}T.
	\end{equation*}
	\item[ii)] If $X$ is a non-isotropic vector field along the $
	{\alpha}({x})$, then Fermi-Walker derivative of $X$ is given by 
		\begin{equation*}
	\widetilde{\nabla }_{{T}}{X}=\nabla _{{T}}%
	{X}-\kappa{\langle T, X\rangle}N.
	\end{equation*}
		\end{enumerate}
\end{lem}

\begin{proof}
	Using definition \ref{fermi dfn} and the equation \eqref{galpro}, the above equations are obtained.
\end{proof}
\begin{cor}
		Let ${X}$ be an isotropic vector field along the curve ${\alpha}$. Then, the Fermi-Walker derivative coincides with Fermi derivative. 
\end{cor}
\begin{cor}
	Let ${X}$ be a non-zero isotropic vector field along the curve ${\alpha}(x)$ which is not a line. Fermi--Walker derivative coincides with derivative of ${X}$ if and only if the vector field ${X}$ is linearly dependent with the binormal vector field ${B}.$
\end{cor}
\begin{proof}
	Using lemma \ref{L1}(i) and $X=\mu N+\lambda B$ 
$$\widetilde{\nabla }_{{T}}{X}=\nabla _{{T}}{X}+\mu\kappa T\notag$$ is obtained. Therefore, $\widetilde{\nabla }_{{T}}{X}=\nabla _{{T}}{X}$ iff $\mu=0.$
\end{proof}

\begin{cor}
	Let ${X}$ be a non-isotropic vector field along the curve ${\alpha}(x)$. Then, Fermi--Walker derivative is not coincide with derivative of ${X}.$
\end{cor}

\begin{thm}\label{T1}
	Let ${\alpha }$ be a curve in $\mathbb{G}^{3},
	{X}={\lambda }_{1}{T}+{\lambda }_{2}{N}+{\lambda }_{3}{B}$ be any non-isotropic vector field
	along ${\alpha}.$ The vector field ${X}$ is
	Fermi--Walker transported along the curve ${\alpha }$ if and only if
	\begin{align*}
	{\lambda }_{1}({x})& =\text{const,} \\
	{\lambda }_{2}({x})& ={c}_{1}\cos\	\Big(\int\limits_{1}^{{x}}{\tau }({t})d{t}\Big)+{c}_{2}\sin \Big(\int\limits_{1}^{{x}}{\tau }(
	{t})d{t}\Big) \\
	{\lambda }_{3}({x})& ={c}_{2}\cos\
	\Big(\int\limits_{1}^{{x}}{\tau}({t})d{t}\Big)-{c}_{1}\sin \Big(\int\limits_{1}^{{x}}{\tau }(
	{t})d{t}\Big)
	\end{align*}
	where ${c}_{1},{c}_{2}$ are constants of integration and ${
		\lambda }_{1}$,${\lambda }_{2}$,${\lambda }_{3}$ are
	continuously differentiable functions of arc length parameter $%
	{x}$.
\end{thm}

\begin{proof}
	By the lemma \ref{L1} ii), 
	\begin{equation*}
	\widetilde{\nabla}_{{T}}{X}=\Big(\frac{d{\lambda }_{1}}{
		d{x}}\Big){T}+\Big(\frac{d{\lambda }_{2}}{d{x
	}}-{\tau }{\lambda }_{3}\Big){N}+\Big(\frac{d
		{\lambda }_{3}}{d{x}}+{\tau }{
		\lambda }_{2}\Big){B}
	\end{equation*}
	is obtained. ${X}$ is Fermi--Walker transported along the curve iff 
	\begin{align*}
	\frac{d{\lambda }_{1}}{d{x}}& =0, \\
	\frac{d{\lambda }_{2}}{d{x}}-{\tau }{
		\lambda }_{3}& =0, \\
	\frac{d{\lambda }_{3}}{d{x}}+{\tau }{\lambda }_{2}& =0.
	\end{align*}
	From the solution of the equation system, 
	\begin{align*}
	{\lambda }_{1}& =const., \\
	{\lambda }_{2}& ={c}_{1}\cos \Big(\int\limits_{1}^{
		{x}}{\tau }({t})d{t}\Big)+{c}
	_{2}\sin \Big(\int\limits_{1}^{{x}}{\tau }({t})d{t}\Big), \\
	{\lambda }_{3}& ={c}_{2}\cos \Big(\int\limits_{1}^{
		{x}}{\tau }({t})d{t}\Big)-{c}_{1}\sin \Big(\int\limits_{1}^{{x}}{\tau }({t})d{t}\Big).
	\end{align*}
	The rest is obvious.
\end{proof}
\begin{cor} 
		Let ${X}=
	{\lambda }_{1}{T}+{\lambda }_{2}{N}+
	{\lambda }_{3}{B}$ be any non-isotropic vector field along $
	{\alpha }$ and the parameters ${\lambda}_{i}$ are
	constants. The vector field ${X}$ is Fermi--Walker transported if and only if the curve $\alpha$ is the planar curve or $\lambda_2=\lambda_3=0.$
\end{cor}

\begin{proof}
	Using  Theorem \ref{T1} and 
	$\forall {\lambda }_{i}=const.,$%
	\begin{equation*}
	\widetilde{\nabla }_{{T}}{X}={\tau }\left( -{\lambda }_{3}{N}+{
		\lambda }_{2}{{B}}\right)
	\end{equation*}%
	is obtained. Therefore, the proof is clear.
\end{proof}




\begin{thm}
	\label{T2}Let ${\alpha }$ be a curve in $\mathbb{G}^{3},
	{X}={\lambda }_{2}{N}+{\lambda }_{3}{B}$ be any non-zero isotropic vector field
	along ${\alpha}.$ The vector field ${X}$ is
	Fermi--Walker transported along the curve ${\alpha }$ if and only if
	\begin{align*}
	{\lambda }_{2}&\kappa = 0 \\
	{\lambda }_{2}& ={c}_{1}\cos\	\Big(\int\limits_{1}^{{x}}{\tau }({t})d{t}\Big)+{c}_{2}\sin \Big(\int\limits_{1}^{{x}}{\tau }(
	{t})d{t}\Big) \\
	{\lambda }_{3}& ={c}_{2}\cos\
	\Big(\int\limits_{1}^{{x}}{\tau}({t})d{t}\Big)-{c}_{1}\sin \Big(\int\limits_{1}^{{x}}{\tau }(
	{t})d{t}\Big)
	\end{align*}
	where ${c}_{1},{c}_{2}$ are constants of integration and ${\lambda }_{2}$,${\lambda }_{3}$ are
	continuously differentiable functions of arc length parameter $%
	{x}$.
\end{thm}

\begin{proof}
	By the lemma \ref{L1} i), 
	\begin{equation*}
	\widetilde{\nabla}_{{T}}{X}=\Big({\lambda }_{2}{\kappa}\Big){T}+\Big(\frac{d{\lambda }_{2}}{d{x
	}}-{\tau }{\lambda }_{3}\Big){N}+\Big(\frac{d
		{\lambda }_{3}}{d{x}}+{\tau }{
		\lambda }_{2}\Big){B}
	\end{equation*}
	is obtained. ${X}$ is Fermi--Walker transported along the curve iff 
	\begin{align*}
	{\lambda }_{2}{\kappa}& =0, \\
	\frac{d{\lambda }_{2}}{d{x}}-{\tau }{
		\lambda }_{3}& =0, \\
	\frac{d{\lambda }_{3}}{d{x}}+{\tau }{\lambda }_{2}& =0.
	\end{align*}
	This is equivalent to 
	\begin{align*}
	{\lambda }_{2}&\kappa =0, \\
	{\lambda }_{2}& ={c}_{1}\cos \Big(\int\limits_{1}^{
		{x}}{\tau }({t})d{t}\Big)+{c}
	_{2}\sin \Big(\int\limits_{1}^{{x}}{\tau }({t})d{t}\Big), \\
	{\lambda }_{3}& ={c}_{2}\cos \Big(\int\limits_{1}^{
		{t}}{\tau }({t})d{t}\Big)-{c}_{1}\sin \Big(\int\limits_{1}^{{x}}{\tau }({t})d{t}\Big).
	\end{align*}
	The rest is obvious.
\end{proof}

\begin{cor} 
	Let ${\alpha }$ be a curve in $\mathbb{G}^{3}$ and
	${X}={\lambda }_{2}{N}+{\lambda }_{3}{B}$ be any non-zero isotropic vector field
	along ${\alpha}.$ 
	\item[i)] If ${\alpha }$ is a line in Galilean space, then the vector field $X$ is Fermi-Walker transported.
	
	\item[ii)] If ${\alpha }$ is a planar curve which is not a line, and ${\lambda}_{2}=0$ then the vector field $X$ is Fermi-Walker transported.
\end{cor}


\begin{cor}
	Let $\{{T},{N},{B}\}$ be the Frenet frame
	of ${\alpha }$. The $\{{T},{N},{B}\}$ is a non-rotating frame along the curve if and only if the curve is a line. Otherwise, the Frenet frame is not a non-rotating frame along the curve in Galilean space.
\end{cor}


%
%
%
%
%
\section{Darboux Frame and Fermi--Walker Derivative in Galilean Space}
Frame fields constitute a very useful tool for studying curves and surfaces. However, the Frenet frame ${T, N, B}$ of $\alpha$ is not useful to describe the geometry of surface $M$. Since $N$ and $B$ in general will be neither tangent nor perpendicular to M. Therefore, we require another frame of $\alpha$ for study the relation between the geometry of $\alpha$ and $M$. There is such a frame field that is called Darboux frame field of $\alpha$ with respect to $M$. The Darboux frame field consists of the triple of vector fields ${T, Q, n}$. The first and last vector fields of this frame $T$ and $n$ are a unit tangent vector field of $\alpha$ and unit normal vector field of $M$ at the point $\alpha(x)$ of $\alpha$. Let $Q=n\times_{G}T$ be the tangential-normal. 


\begin{thm}Let  $\alpha :I\subset \mathbb{R}\rightarrow M\subset \mathbb{G}^{3}$ be a unit-speed curve, and let {T, Q, n} be the Darboux frame field of $\alpha$ with respect to M. Then
	\begin{equation}\label{Darboux}
	\begin{bmatrix}
	T \\ 
	Q \\ 
	n
	\end{bmatrix}
	^{\prime }=
	\begin{bmatrix}
	0 & \kappa_g & \kappa_n \\ 
	0 & 0 & \tau_g \\ 
	0 & -\tau_g & 0
	\end{bmatrix}
	\begin{bmatrix}
	T \\ 
	Q \\ 
	n
	\end{bmatrix}\,.
	\end{equation}
	where $\kappa_g$ and $\kappa_n$ give the tangential and normal component of the curvature vector, and these functions are called the geodesic and the normal curvature, respectively \cite{sahin}.
\end{thm}

\begin{proof}	We have 
	\begin{equation}\label{c}
	\begin{split}
	T'&= (T'\cdot_{G}Q)Q+(T'\cdot_{G}n)n
	\\&= (\alpha''\cdot_{G}Q)Q+(\alpha''\cdot_{G}n)n
	\\&=\kappa_g Q + \kappa_n n.
	\end{split}
	\end{equation}
	The other formulae are proved in a similar fashion.
\end{proof}

Also, (2.7) implies the important relations
\begin{equation}\label{kt}
\kappa^2(x)=\kappa^2_g(x)+\kappa^2_n(x),  \hskip .5cm \tau(x)=-\tau_g(x)+\frac{\kappa'_g(x)\kappa_n(x)-\kappa_g(x)\kappa'_n(x)}{\kappa^2_g(x)+\kappa^2_n(x)}
\end{equation} 
where $\kappa^2(x)$ and $\tau(x)$ are the square curvature and the torsion of $\alpha$, respectively.
\begin{lem}
	\label{L2}Let ${\alpha}:I\subset \mathbb{R}\rightarrow\mathbb{G}^{3}$ be a curve in $\mathbb{G}^{3}$ and ${X}$ is any vector field along the curve, Fermi-Walker derivative with respect to the Darboux frame can be expressed as
	\begin{enumerate}
		\item[i)] If $X$ is an isotropic vector field along $
		{\alpha}$, then Fermi-Walker derivative with respect to the Darboux frame of $X$ is given by 
		\begin{equation*}
		\widetilde{\nabla }_{{T}}{X}=\nabla _{{T}}{X}+\big(\kappa_g{\langle Q, X\rangle}+\kappa_n{\langle n, X\rangle}\big)T.
		\end{equation*}
		\item[ii)] If $X$ is a non-isotropic vector field along $
		{\alpha}$ in $\mathbb{G}^{3}$, then Fermi-Walker derivative with respect to the Darboux frame of $X$ is given by 
		\begin{equation*}
		\widetilde{\nabla }_{{T}}{X}=\nabla _{{T}}{X}-\big(\kappa_g Q+\kappa_n n\big){\langle T, X\rangle}.
		\end{equation*}
	\end{enumerate}
\end{lem}

\begin{proof}
	Using definition \ref{fermi dfn} and the equation \eqref{galpro}, the above equations are obtained.
\end{proof}
\begin{cor}
	Let ${X}=\lambda_2 Q+\lambda_3 n $ be an isotropic vector field along the curve ${\alpha}$. The Fermi-Walker derivative coincides with derivative of $X$ iff $\lambda_2 \kappa_g+\lambda_3 \kappa_n=0$. 
\end{cor}
\begin{cor}
	Let ${X}$ be a non-isotropic vector field along the curve ${\alpha}$. The Fermi--Walker derivative coincides with derivative of ${X}$ iff $\alpha$ is a line in $\mathbb{G}^{3}$.
\end{cor}
\begin{proof}
	Using lemma \ref{L2}(ii) and $X=\lambda_1T+\lambda_2 Q+\lambda_3 n$,  
	$$\widetilde{\nabla }_{{T}}{X}=\nabla _{{T}}{X}-\lambda_1(\kappa_g Q+\kappa_n n)\notag$$ is obtained. Since $X$ is a non-isotropic, $\lambda_1\neq 0$. Therefore,  $\widetilde{\nabla }_{{T}}{X}=\nabla _{{T}}{X}$ iff $\kappa_g Q+\kappa_n n=0.$ Hence, $\kappa_g=\kappa_n=0$. That is, the curve is a line in $\mathbb{G}^{3}.$ 
\end{proof}

\begin{thm}\label{T2}
	Let ${\alpha }$ be a curve in $\mathbb{G}^{3},
	{X}={\lambda }_{1}{T}+{\lambda }_{2}{N}+{\lambda }_{3}{B}$ be any non-isotropic vector field
	along ${\alpha}.$ The vector field ${X}$ is
	Fermi--Walker transported along the curve ${\alpha }$ if and only if
	\begin{align*}
	{\lambda }_{1}({x})& =\text{const,} \\
	{\lambda }_{2}({x})& ={c}_{1}\cos\	\Big(\int\limits_{1}^{{x}}{\tau_g }({t})d{t}\Big)+{c}_{2}\sin \Big(\int\limits_{1}^{{x}}{\tau_g }(
	{t})d{t}\Big) \\
	{\lambda }_{3}({x})& ={c}_{2}\cos\
	\Big(\int\limits_{1}^{{x}}{\tau_g}({t})d{t}\Big)-{c}_{1}\sin \Big(\int\limits_{1}^{{x}}{\tau_g }(
	{t})d{t}\Big)
	\end{align*}
	where ${c}_{1},{c}_{2}$ are constants of integration and ${
		\lambda }_{1}$,${\lambda }_{2}$,${\lambda }_{3}$ are
	continuously differentiable functions of arc length parameter $%
	{x}$.
\end{thm}

\begin{proof}
	By the lemma \ref{L2} ii), 
	the proof is obvious.
\end{proof}
\begin{cor}
	Let ${X}=
	{\lambda }_{1}{T}+{\lambda }_{2}{N}+
	{\lambda }_{3}{B}$ be any non-isotropic vector field along $
	{\alpha }$ and the parameters ${\lambda}_{i}$ are
	constants. The vector field ${X}$ is Fermi--Walker transported iff the curve $\alpha$ is the line of curvature or $\lambda_2=\lambda_3=0.$
	
\end{cor}

\begin{proof}
	Using  lemma \ref{L2}(ii) and 
	$\forall {\lambda }_{i}=const.,$%
	\begin{equation*}
	\widetilde{\nabla }_{{T}}{X}={\tau_g }\left( {\lambda }_{2}{n}+{
		-\lambda }_{3}{{Q}}\right)
	\end{equation*}%
	is obtained. Using the above equation, the proof can be obtained.
\end{proof}




\begin{thm}
	\label{T2}Let ${\alpha }$ be a curve in $\mathbb{G}^{3},
	{X}={\lambda }_{2}{Q}+{\lambda }_{3}{n}$ be any non-zero isotropic vector field
	along ${\alpha}.$ The vector field ${X}$ is
	Fermi--Walker transported along the curve ${\alpha }$ if and only if
	\begin{align*}
	{\lambda }_{2}&\kappa_g+\lambda_3\kappa_n = 0 \\
	{\lambda }_{2}& ={c}_{1}\cos\	\Big(\int\limits_{1}^{{x}}{\tau_g }({t})d{t}\Big)+{c}_{2}\sin \Big(\int\limits_{1}^{{x}}{\tau_g }(
	{t})d{t}\Big) \\
	{\lambda }_{3}& ={c}_{2}\cos\
	\Big(\int\limits_{1}^{{x}}{\tau_g}({t})d{t}\Big)-{c}_{1}\sin \Big(\int\limits_{1}^{{x}}{\tau_g }(
	{t})d{t}\Big)
	\end{align*}
	where ${c}_{1},{c}_{2}$ are constants of integration and ${\lambda }_{2}$,${\lambda }_{3}$ are
	continuously differentiable functions of arc length parameter $%
	{x}$.
\end{thm}

\begin{proof}
	By the lemma \ref{L2} i), the results are obvious.
\end{proof}

\begin{cor} 
	Let ${\alpha }$ be a curve in $\mathbb{G}^{3},
	{X}={\lambda }_{2}{N}+{\lambda }_{3}{B}$ be any non-zero isotropic vector field
	along ${\alpha}$ and the parameters ${\lambda}_{i}$ are
	constants. The vector field $X$ along the curve ${\alpha }$ in $\mathbb{G}^3$ is the Fermi-Walker transported iff $$\Big(\frac{\kappa_g}{\kappa_n}\Big)'=0.$$
%
\end{cor}


\begin{cor}
	Let $\{{T},{Q},{n}\}$ be the Darboux frame
	of the curve ${\alpha }$. $\{{T},{Q},{n}\}$ Darboux frame of the curve is a non-rotating frame if and only if the curve is a line. Otherwise, the Darboux frame is not a non-rotating frame along the curve in Galilean space.
\end{cor}

\section{Conclusions}
The notion of Fermi-Walker derivative, it shows us one method, which is used
for defining \textquotedblright constant\textquotedblright direction, that
may contain lots of condition to have Fermi-Walker transport or non-rotating
frame. The condition of Fermi-Walker transport depends on a solution that
contains differential equation system which is not always easy to find the
answer. Therefore, it is important to analyze this concept. In this paper,
Fermi-Walker derivative, Fermi-Walker transport and non-rotating frame
concepts are defined along any curve and the notions have been analyzed or
both isotropic and non-isotropic vector fields.

We have investigated Fermi-Walker derivative and geometric applications in
various spaces like Euclidean, Lorentz and Dual space up to now \cite%
{7,16,17}. The Fermi-Walker derivative which is defined in $\mathbb{G}^{3}$
is different from them so far since it is examined for both isotropic and
non-isotropic vector fields.

Firstly, Fermi-Walker derivative is redefined for any isotropic vector
fields along a curve which is in Galilean space. We have proved Fermi-Walker
derivative that is defined for the isotropic vector fields coincides with
Fermi derivative which is defined in any surface. We have shown Fermi-Walker
derivative which is defined for any non-isotropic vector fields is not
coincides with derivative of the vector fields. Being Fermi-Walker transport
conditions are examined for any isotropic and non-isotropic vector fields.
We have shown that if the curve is a line or a planar curve which is not a
line then the non-zero isotropic vector field is Fermi-Walker transported.
We have obtained that Frenet frame is not a non-rotating frame if the curve
is not a line. 

Then, similar investigations have been made for any isotropic and
non-isotropic vector fields with respect to the Darboux frame in Galilean
space. We have proved while the curve is a line the Darboux frame is a
non-rotating frame.

\end{document}